\documentclass[final]{siamltex}
\usepackage{amscd,amssymb,epsfig,mathtools}
\usepackage{times}
\usepackage[letterpaper=true,colorlinks=true,
        linkcolor=red,filecolor=green,citecolor=red,
        pdfpagemode=None]{hyperref}

\newcommand{\diam}{\operatorname{diam}}

\newcommand{\ud}{\,d} 
 
\newcommand{\R}{\mathbb{R}}

\newcommand{\tir}[1]{\ensuremath{\overline {#1}}} 
\newtheorem{thm}{Theorem}[section]

\newtheorem{prop}[thm]{Proposition} 
 
\newtheorem{defn}[thm]{Definition}

\def\whsq{\vbox to 5.8pt 
{\offinterlineskip\hrule 
\hbox to 5.8pt{\vrule height 
5.1pt\hss\vrule height 5.1pt}\hrule}}

\def\<{\langle} 
\def\>{\rangle} 
\def\PP{{\mathop{{\rm I}\kern-.2em{\rm P}}\nolimits}} 
\def\FF{{\mathop{{\rm I}\kern-.2em{\rm F}}\nolimits}}   
\def\ZZ{{\mathop{{\rm I}\kern-.2em{\rm Z}}\nolimits}}

\title{Discrete Aleksandrov solutions of the Monge-Amp\`ere equation}

\author{Gerard Awanou\thanks{Department of Mathematics, Statistics, and Computer Science, M/C 249.
University of Illinois at Chicago, 
Chicago, IL 60607-7045, USA ({\tt awanou@uic.edu}).}
        }

\begin{document}

\maketitle

\begin{abstract}
We prove the convergence of a wide stencil finite difference scheme to the Aleksandrov solution of the elliptic Monge-Amp\`ere equation when the right hand side is a sum of Dirac masses. The discrete scheme we analyze for the Dirichlet problem, when coupled with a discretization of the second boundary condition, can be used to get a good initial guess for geometric methods solving optimal transport between two measures. 

\end{abstract}



\pagestyle{myheadings}
\thispagestyle{plain}
\markboth{G. AWANOU }{Discrete Aleksandrov solutions of the Monge-Amp\`ere equation}

\section{Introduction}

In this paper we prove the convergence of a wide stencil finite difference scheme to the weak solution, in the sense of Aleksandrov, of the Dirichlet problem for the Monge-Amp\`ere equation
\begin{align} \label{m1}
\begin{split}
\det D^2 u &=\sum_{l=1}^N c_l \delta_{d_l} \, \text{in} \, \Omega\\
 u &=g \, \text{on} \, \partial \Omega,
 \end{split}
\end{align}
where $d_l$ is a point inside $\Omega$, $c_l$ a real number, $\delta_{d_l}$ the Dirac mass at $d_l$ and $N$ is the number of Dirac masses.
Here $\Omega$ is a convex bounded domain of $\R^d$ with boundary $\partial \Omega$. It is assumed that 
 $g \in C(\partial \Omega)$ can be extended to a  convex function $\tilde{g} \in C(\tir{\Omega})$. The domain is not assumed to be strictly convex. The discretization we analyze is a variant of the one proposed in \cite{BenamouFroese2017} for the second boundary value problem. The latter can be used to get a good initial guess for geometric methods solving optimal transport between two measures. 
 Our scheme uses a discretization of the subdifferential at the mesh points $d_l$ and a standard wide stencil scheme at other points \eqref{m1h-f}. At the mesh points $d_l$, our scheme requires the computation of areas of polygons as in \cite{Mirebeau15}.
 
 While there are several convergence analysis of numerical schemes for the second boundary condition \cite{LindseyRubinstein,benamou2017minimal,Froese2019}, they address the case where the right hand side of the equation is absolutely continuous with respect to the Lebesgue measure. Here we consider a right hand side which is a sum of Dirac masses but with the Dirichlet boundary condition. The scheme we analyze, leads to a set function that overestimates the discrete Monge-Amp\`ere measure defined through a discrete version of the subdifferential. This allows us to use essentially the same tools as in the Aleksandrov theory of \eqref{m1}. We do not address convergence rates for which we refer for example to \cite{ChenHuangWang2019} for a study of the Oliker-Prussner discretization and the Dirichlet problem and to \cite{berman2018convergence} for the second boundary condition. We note that the convergence analysis cannot be done in the framework of viscosity solutions since the right hand side of \eqref{m1} involves Dirac masses. On the other hand the convergence does not follow from results available for Aleksandrov solutions since our approximations are not convex functions. A similar difficulty arised in \cite{LindseyRubinstein,benamou2017minimal}.
 

The convergence analysis we give relies on recent results on uniform limits of discrete convex functions and the weak convergence of the associated discrete Monge-Amp\`ere measures \cite{awanou2019uniform,awanou2019uweakcvg}.

 
The paper is organized as follows. In the next section we collect some notation used throughout the paper and 
recall the notion of Aleksandrov solution. 
We also recall the characterization of the subgradient which forms the basis of the discretization proposed in \cite{BenamouFroese2017}. Since the latter involves a quadrature rule, it may not lead to a one sided approximation of the subdifferential, a property which allows the use of tools from the Aleksandrov theory of \eqref{m1}. We then introduce a partial Monge-Amp\`ere measure associated to a mesh function. The resulting scheme is readily analyzed in  section \ref{cvg-analysis}. We conclude with some remarks.

\section{Preliminaries}  \label{as}

We use the notation $|| . ||$ for the Euclidean norm of $\R^d$. Let $h$ be a small positive parameter and let
$$
\mathbb{Z}_h^d = \{\, m h, m \in \mathbb{Z}^d \, \},
$$
denote the orthogonal lattice with mesh length $h$. We define
\begin{equation} \label{interior-domain}
\Omega_h =  \Omega \cap \mathbb{Z}^d_h \ \text{and} \  \partial \Omega_h =  \partial \Omega \cap \mathbb{Z}^d_h,
\end{equation}
and denote by $\mathcal{U}_h$ the linear space of mesh functions, i.e. real-valued functions defined on 
$$
\tir{\Omega}_h := \Omega_h \cup \partial \Omega_h.
$$
For $x \in \Omega_h$, $e \in \mathbb{Z}^d, e \neq 0$ such that $x \pm h e \in \tir{\Omega}_h$ and $v_h \in \mathcal{U}_h$, let
$$
\Delta_e v_h (x) = v_h(x+ h e ) - 2 v_h(x) + v_h(x- h e ) .
$$

\begin{defn} \label{discrete-convex}
We say that a mesh function $v_h$ is {\it discrete convex} if and only if $\Delta_e v_h(x) \geq 0$ for all $x \in \Omega_h$ and $e \in \mathbb{Z}^d$.
\end{defn}

\noindent
We denote by $\mathcal{C}_h$ the cone of discrete convex mesh functions. The points $d_l, l=1,\ldots,N$ are assumed to be mesh points.

\subsection{Aleksandrov solutions}  \label{as}
The material in 
this subsection is taken from  \cite{Guti'errez2001} to which we refer for proofs. Let $\Omega$ be an open subset of $\R^d$ and let us denote by $\mathcal{P}(\R^d)$ the set of subsets of $\R^d$. 
\begin{defn}
Let $u: \Omega \to \R$.  The normal mapping of $u$, or subdifferential of $u$ is the set-valued mapping $\partial u: \Omega \to \mathcal{P}(\R^d)$ defined by
\begin{align} \label{normal-mapping}
\partial u (x_0) = \{ \, p \in \R^d: u(x) \geq u(x_0) + p \cdot (x-x_0), \, \text{for all} \, x \in \Omega\,\}.
\end{align}
\end{defn}
Let $|E|$ denote the Lebesgue measure of the measurable subset $E \subset \Omega$. For $E \subset \Omega$, we define 
$$
\partial u(E) = \cup_{x \in E} \partial u(x).
$$ 

\begin{thm}[\cite{Guti'errez2001} Theorem 1.1.13]
If $u$ is continuous on $\Omega$, the class
 \begin{align*}
\mathcal{S} = \{\, E \subset \Omega, \partial u(E) \, \text{is Lebesgue measurable}\, \},
\end{align*}
is a Borel $\sigma$-algebra and the set function $M [u]: \mathcal{S}  \to \tir{\R}$ defined by
$$
M [u] (E) = |\partial u(E)|,
$$
is a measure, finite on compact subsets, called the Monge-Amp\`ere measure associated with the function $u$. 
\end{thm}

\noindent
We can now define the notion of Aleksandrov solution of the Monge-Amp\`ere equation. 
\begin{defn}
Let $\Omega \subset \mathcal{P}(\R^d)$ be open and convex. Given a Borel measure $\nu$ on $\Omega$, a convex function $u \in C(\Omega)$ is an Aleksandrov solution of
$$
\det D^2 u = \nu,
$$
if the associated Monge-Amp\`ere measure $M[u]$ is equal to $\nu$.
\end{defn}

\noindent
We recall an existence and uniqueness result for the solution of \eqref{m1}.

\begin{prop}  [\cite{Hartenstine2006} Theorem 1.1] \label{ex-Alex}
Let $\Omega$ be a bounded convex domain of $\R^d$. Assume $\nu$ is a finite Borel measure and $g \in C(\partial \Omega)$ can be extended to a function $\tilde{g} \in C(\tir{\Omega})$ which is  convex in $\Omega$. Then the Monge-Amp\`ere equation \eqref{m1}
has a unique convex Aleksandrov solution in $C(\tir{\Omega})$.

\end{prop}

\begin{defn}
A sequence $\mu_n$ of Borel measures converges to a Borel measure $\mu$ if and only if $\mu_n(B) \to \mu(B)$ for any Borel set $B$ with $\mu(\partial B)=0$.
\end{defn}

\noindent
We note that there are several equivalent definitions of weak convergence of measures which can be found for example in \cite[Theorem 1, section 1.9]{Evans-Gariepy}.

\subsection{Description of the scheme obtained through quadrature} \label{quad}

In two dimension, it was shown in \cite{BenamouFroese2017} that the  subdifferential can be written in angular terms using  one-sided directional derivatives. Let $e_{\theta} = (\cos \theta, \sin \theta)$ and put
$$
\partial_{\theta} u(x_0)= \lim_{r \to 0^+} \frac{u(x_0 + r \theta) - u(x)}{r}. 
$$
Recall that for $r \in \R, r^+= \max (r,0) $. We have from \cite{BenamouFroese2017} 
\begin{thm} Let $u$ be a convex function on $\Omega \subset \R^2$. Then at any point $x_0 \in  \Omega$
\begin{equation*} 
|\partial u (x_0)| = \int_0^{2 \pi} \frac{1}{2} \bigg( R^+[u](x_0, \theta)^2 - R^-[u](x_0, \theta)^2 \bigg)^+ d \theta,
\end{equation*}
where
\begin{align*}
R^+[u](x_0, \theta) & = \inf_{\theta' \in (\theta - \frac{\pi}{2},\theta + \frac{\pi}{2} )} \frac{\big( \partial_{\theta'} u(x_0) \big)^+}{ \cos(\theta - \theta')} \\
R^-[u](x_0, \theta) & = \sup_{\theta' \in (\theta - \frac{\pi}{2},\theta + \frac{\pi}{2} )} \frac{\big( - \partial_{\theta'  + \pi } u(x_0) \big)^+}{ \cos(\theta - \theta')}.
\end{align*}
\end{thm}
The numerical scheme proposed in \cite{BenamouFroese2017}  is based on the following observation.

Let $\theta'_j, j=1,\ldots,n$ denote a set of angular directions such that $e_j=|e_j| (\cos \theta_j', \sin \theta_j')$ is the vector of smallest length such that both $x\pm e_j \in \tir{\Omega} \cap\mathbb{Z}^2_h$. We may assume that all $\theta'_j$
are in an interval of length $\pi$. 

For a mesh function $u_h$,  define
\begin{align*}
R^{-}_h[u_h](x,\theta) & = \sup_{j=1,\ldots,n} \frac{u_h(x)-u_h(x-e_j)}{|e_j| \cos(\theta - \theta'_j) } \\
R^{+}_h[u_h](x,\theta) & = \inf_{j=1,\ldots,n} \frac{u_h(x+e_j)-u_h(x)}{|e_j| \cos(\theta - \theta'_j) }. 
\end{align*}
Next, note that $\{ \, \theta_j, j=1,\ldots,n\, \} \cup \{ \,  \theta_j + \pi, j=1,\ldots,n\, \}$ form a partition $\eta_k, k=1,\ldots,M$ of $[0,2 \pi)$. We then have for $x_0 \in \Omega_h$, with $\theta_{M+1} = \theta_1$ and $u_h$ the restriction of $u$ to $\tir{\Omega}_h$
\begin{equation*}
|\partial u (x_0)|  \approx \sum_{k=1}^{M}  \frac{1}{2} (\theta_{k+1} - \theta_k)  \max \big(R^{+}_h[u_h](x_0,\theta_k)^2 - \max\{ R^{-}_h[u_h](x_0,\theta_k),0\}^2 , 0 \big).
\end{equation*}
We propose below a more accurate approximation of the area of the subdifferential, which does not rely on a quadrature rule, namely
\begin{equation}  \label{sub-1h-sided}
\int_0^{2 \pi} \frac{1}{2} \big(R^{+}_h[u_h](x_0,\theta)^2 - \max\{ R^{-}_hu_h](x_0,\theta),0\}^2 \big) \ud \theta,
\end{equation}
can be computed as the area of a polygon as in \cite{Mirebeau15}, and involves only a finite number of mesh points close to $x_0$.
The motivation of the modified scheme is that it has a natural connection with the discrete Monge-Amp\`ere measure defined through a discrete version of the subdifferential. We recall that for the approximation of \eqref{m1}, this scheme will be used only at the points $d_l$ and a standard wide stencil scheme used at all other points. In practice, it may be necessary not to use all directions required in \eqref{sub-1h-sided}. This will introduce an additional discretization error. For convergence, it is necessary to take both $h \to 0$ and use all directions required in \eqref{sub-1h-sided}.

\subsection{Discretizations of the normal mapping}

For 
a mesh function $u_h \in \mathcal{C}_h$, the discrete normal mapping of $u_h$ at the point $x \in \Omega \cap\mathbb{Z}^d_h$ 
is defined as
\begin{align*}
\partial_h u_h(x) = \{ \, p \in \R^d, p \cdot e \geq u_h(x) - u_h(x-e) \, \forall  e \in \mathbb{Z}^d_h \text{ such that } 
x - e \in \tir{\Omega} \cap\mathbb{Z}^d_h \, \}.
\end{align*}
For convenience, we will often omit the mention that we need $x - e \in \tir{\Omega} \cap\mathbb{Z}^d_h$ in the definition of $\partial_h u_h (x )$.

For a subset $E \subset \Omega$, we define 
$$
\partial_h u_h(E) = \cup_{x \in E \cap  \mathbb{Z}^d_h} \partial_h u_h(x),
$$
and 
\begin{equation*}
M_h[u_h] (E) = | \partial_h u_h(E)  | \text{ for a Borel set } E.
\end{equation*}
Note that for $|E|$ sufficiently small and $x \in E$, we have $M_h[u_h] (E)=M_h[u_h] (\, \{ x \, \})$. We will make the abuse of notation
$$
M_h[u_h] ( x )  = M_h[u_h] (\, \{ x \, \}).
$$

Let us now consider the following discrete analogue of the normal mapping. We define
\begin{align*}
\partial_h^1 u_h (x ) & = \{ \, p  \in \R^d: \forall  e \in \mathbb{Z}^d_h,  u_h(x)-u_h(x-e) \leq p \cdot e \leq u_h(x+e)-u_h(x), \\
& \,  \text{provided} \, x\pm e \in \tir{\Omega} \cap\mathbb{Z}^d_h \,\}.
\end{align*}
We then define
$$
M_h^1[u_h] ( x )  = |\partial_h^1 u_h (x )|.
$$
Clearly
$$
\partial_h u_h (x ) \subset \partial_h^1 u_h (x ),
$$
and thus for $x \in \Omega_h$
\begin{equation} \label{rel1}
M_h[u_h] ( x ) \leq M_h^1[u_h] ( x ),
\end{equation}
that is, $M_h^1[u_h] ( x )$ overestimates the ''true'' discrete Monge-Amp\`ere measure $M_h[u_h] ( x )$.

\begin{thm} 
In dimension $d=2$ and for $x_0 \in \Omega_h$, $M_h^1[u_h] ( x_0 )$ is given by \eqref{sub-1h-sided}.
\end{thm}

\begin{proof} We use the same notation as in section \ref{quad}. Let $e \in \mathbb{Z}^2_h, e \neq 0$ such that $x_0\pm e \in \tir{\Omega} \cap\mathbb{Z}^2_h$. Put $e=|e| (\cos \theta', \sin \theta')$ and note that $-e=|e| (\cos \theta'+\pi, \sin \theta'+\pi)$. 

The condition 
$u_h(x_0)-u_h(x_0-e) \leq p \cdot e \leq u_h(x_0+e)-u_h(x_0) \ \forall  e \in \mathbb{Z}^d_h$ is equivalent to $u_h(x_0)-u_h(x_0-(-e)) \leq p \cdot (-e) \leq u_h(x_0+(-e))-u_h(x_0) \ \forall  e \in \mathbb{Z}^d_h$. Thus we may restrict $\theta'$ to be in an interval of length $\pi$.

We prove that if $x_0 + r e_j 
\in \tir{\Omega} \cap\mathbb{Z}^2_h$, then $r$ must be an integer. Put $e_j=(k h,m h)$ for integers $k$ and $m$. Then $r k = k'$ and $r m =m'$ for integer $k'$ and $m'$. Thus $r$ must be a rational number. Assume 
$r=a/b$ with $a$ and $b$ having no common divisors. Then $b$ must divide both $k$ and $m$. By the assumption on $e_j$, we conclude that $b=1$ proving that $r$ is an integer. 

Next, since $u_h \in \mathcal{C}_h$, the condition $u_h(x_0)-u_h(x_0-e_j) \leq p \cdot e_j \leq u_h(x_0+e_j)-u_h(x_0)$ implies $u_h(x_0)-u_h(x_0- 2 e_j) \leq 2 p \cdot  e_j \leq u_h(x_0+ 2 e_j)-u_h(x_0)$ and hence by induction $u_h(x_0)-u_h(x_0- r e_j) \leq r p \cdot  e_j \leq u_h(x_0+r e_j)-u_h(x_0)$.

We can therefore write
\begin{align*}
\partial_h^1 u_h (x_0 ) & = \{ \, p  \in \R^2: \forall  j=1,\ldots,n ,  u_h(x_0)-u_h(x_0-e_j) \\
& \qquad \qquad \qquad \qquad \qquad \qquad \qquad \qquad  \leq p \cdot e_j \leq u_h(x_0+e_j)-u_h(x_0) \,\}.
\end{align*}

Now put $p=r (\cos \theta, \sin \theta), (r,\theta) \in (0,\infty) \times [0,2 \pi)$. Then $p \in \partial_h^1 u_h (x_0 )$ if and only if
\begin{equation} \label{r-theta}
u_h(x_0)-u_h(x_0-e_j) \leq r |e_j| \cos(\theta - \theta'_j) \leq u_h(x_0+e_j)-u_h(x_0), j=1,\ldots,n.
\end{equation}
The set of vectors $p=r (\cos \theta, \sin \theta)$ with $\theta = \theta'_j \pm \pi/2$ for some $j$ form a finite union of lines in $\R^2$ and thus has measure 0. If $p=r (\cos \theta, \sin \theta) \in \partial_h^1 u_h (x_0 )$ with  
$\theta'_j \in (\theta - \pi/2, \theta + \pi/2)$, we obtain from \eqref{r-theta} $R^{-}_h[u_h](x_0,\theta) \leq R^{+}_h[u_h](x_0,\theta)$. We conclude that, up to a set of measure 0
\begin{multline*}
\partial_h^1 u_h (x_0 )  \subset \{ \, p=r (\cos \theta, \sin \theta)  \in \R^2, (r,\theta) \in (0,\infty) \times [0,2 \pi): \forall  j=1,\ldots,n , \\ R^{-}_h[u_h](x_0,\theta) \leq r \leq  R^{+}_h[u_h](x_0,\theta) \,\}.
\end{multline*}
The reverse inclusion is immediate. Thus \eqref{sub-1h-sided} follows.
\end{proof}

We close this section by recalling the standard wide stencil discretization of $\det D^2 u$ \cite{Oberman2010a}.  We define
$$
V = \{ \, (e_1,\ldots,e_d), e_i \in \mathbb{Z}^d, \, i=1,\ldots,d,  (e_1,\ldots,e_d) \, \text{is an orthogonal basis of} \, \R^d \, \},
$$
and a discrete Monge-Amp\`ere operator as
$$
\mathcal{M}_h[v_h](x) = \frac{1}{h^{2d}} \inf_{ (e_1,\ldots,e_d) \in V \atop x\pm h e_i \in \tir{\Omega} \cap \mathbb{Z}^d_h \forall i} \prod_{i=1}^d \frac{\Delta_{e_i} v_h (x)}{|| e_i ||^2}, x \in \Omega_h.
$$
We will also later need the operator
$$
M_h^0[v_h](x) = \frac{1}{h^{d}} \inf_{ (e_1,\ldots,e_d) \in V \atop x\pm h e_i \in \tir{\Omega} \cap \mathbb{Z}^d_h \forall i} \prod_{i=1}^d \frac{\Delta_{e_i} v_h (x)}{|| e_i ||}, x \in \Omega_h.
$$

\begin{thm} \label{compare-all}
Let $\Omega$ be bounded. There exists a constant $C>0$ such that for $x \in \Omega_h$
\begin{equation} \label{comparison3}
M_h[u_h] ( x ) \leq M_h^1 [u_h](x) \leq M_h^0[v_h](x) \leq C  \mathcal{M}_h[u_h](x).
\end{equation}
\end{thm}

\begin{proof} The inequality on the left follows from \eqref{rel1}. We have
\begin{align*}
\partial_h^1 u_h (x ) & = \{ \, p  \in \R^d: \forall  e \in \mathbb{Z}^d,  u_h(x)-u_h(x- h e) \leq h p \cdot e \leq u_h(x+ h e)-u_h(x), \\
& \,  \text{provided} \, x\pm h e \in \tir{\Omega} \cap\mathbb{Z}^d_h \,\}.
\end{align*}
Given $ (e_1,\ldots,e_d) \in V$ and $u_h \in \mathcal{C}_h$, the volume of the set $S_{(e_1,\ldots,e_d)}[u_h](x)$ defined by
$$
\{ \, p  \in \R^d, u_h(x)-u_h(x- h e_i) \leq h p \cdot e_i \leq u_h(x+ h e_i)-u_h(x), i=1,\ldots,d \, \},
$$
is given by 
$$
 \frac{1}{h^{d}}  \prod_{i=1}^d \frac{ \Delta_{e_i}u_h (x) }{ ||e_i||}. 
$$
This follows from the observation that $p \in S_{(e_1,\ldots,e_d)}[u_h](x)$ if and only if
$$
\frac{u_h(x)-u_h(x- h e_i)}{ h ||e_i||} \leq \frac{ p \cdot e_i}{||e_i||} \leq  \frac{u_h(x+ h e_i)-u_h(x)}{h ||e_i||},
$$
and $ (e_1,\ldots,e_d)$ being an orthogonal basis.
Thus since $ \partial_h^1 u_h (x ) \subset S_{(e_1,\ldots,e_d)} [u_h](x)$ we have $| \partial_h^1 u_h (x )| \leq (1/h^d ) \prod_{i=1}^d \Delta_{e_i}u_h (x) /||e_i||$, i.e. $M_h^1 [u_h](x) \leq M_h^0[v_h](x)$. Next,
\begin{align*} 
\begin{split}
 \frac{1}{h^{d}} \prod_{i=1}^d \frac{ \Delta_{e_i} v_h (x)}{|| e_i ||} 
& =  \bigg(h^ d\prod_{i=1}^d || e_i || \bigg)  \frac{1}{h^{2d}} \prod_{i=1}^d \frac{\Delta_{e_i} v_h (x)}{|| e_i ||^2}.
\end{split}
\end{align*}
Since  $x\pm h e_i \in \tir{\Omega} \cap \mathbb{Z}^d_h$ and $\Omega$ is bounded, there exists a constant $C>0$ independent of $i$ such that $|| h e_i|| \leq C $ for all $i$. This implies
that $M_h^1 [u_h](x) \leq C  \mathcal{M}_h[u_h](x)$ and concludes the proof.
\end{proof}

\subsection{The modified scheme}

Dirac masses at mesh points can be approximated using regularized delta functions, see for example \cite[p. 1708]{Oberman2010a}. Let $f_h$ be a sequence of mesh functions which converge weakly to $\sum_{l=1}^N c_l \delta_{d_l} $ as measures, with $d_l, l=1,\ldots,N$ a finite number of given points in $\Omega_h$.

We consider the following discretization of \eqref{m1}: find $u_h \in \mathcal{C}_h$ such that
\begin{align} \label{m1h-f}
\begin{split}
M_h^1 [u_h](x) &=  f_h(x), x \in \cup_{l=1}^N \{ \, d_l \, \} \\
\mathcal{M}_h [u_h](x) &=0, x \in \Omega_h \setminus \cup_{l=1}^N \{ \, d_l \, \}  \\
u_h(x) & = g(x),  x \in \partial \Omega_h.
\end{split}
\end{align}
By construction $f_h(x)=0$ for $x \in \Omega_h \setminus \cup_{l=1}^N \{ \, d_l \, \} $.
We establish the existence, unicity and stability of solutions to \eqref{m1h-f}. We first recall the Brunn-Minkowski's inequality \cite{Schneider14}. 

\begin{lemma} 
For two nonempty, compact convex sets $K$ and $L$, their Minkowski sum is defined as
$$
K+L = \{ \, a+b, a \in K \text{ and } b \in L\, \}.
$$
We have
\begin{equation} \label{Brunn-Minkowski}
| K +  L |^{\frac{1}{d}} \geq  |K|^{\frac{1}{d}} +  |L|^{\frac{1}{d}}.
\end{equation}
\end{lemma}

\begin{lemma} \label{concave}
Given $x \in \Omega_h$ the operator $v_h \to (M_h^1[v_h](x) )^{1/d}$ is concave on $\mathcal{C}_h$.
\end{lemma}

\begin{proof}
We recall that given a set $K$ and $\lambda \in \R$, $\lambda K=\{ \, \lambda x, x \in K\, \}$. We observe that for $\lambda >0$, $p \in\partial_h^1 v_h(x)$ if and only if $\lambda p \in\partial_h^1 (\lambda v_h)(x)$. Thus by the
positive homogeneity (of degree $d$) of volume in $\R^d$
$$
(M_h^1[\lambda v_h](x))^{\frac{1}{d}} = \lambda (M_h^1 [v_h](x))^{\frac{1}{d}}.
$$
It is therefore enough to prove that for $v_h, w_h \in \mathcal{C}_h$, we have 
\begin{equation} \label{conc}
 (M_h^1 [v_h + w_h ](x))^{\frac{1}{d}} \geq  (M_h^1 [v_h](x))^{\frac{1}{d}} +  
 (M_h^1 [w_h](x))^{\frac{1}{d}}. 
\end{equation}

Next, we note that
$$
\partial_h^1 v_h(x) + \partial_h^1 w_h(x) \subset \partial_h^1 (v_h+w_h)(x),
$$
and thus $|\partial_h^1 (v_h+w_h)(x)| \geq |\partial_h^1 v_h(x) + \partial_h^1 w_h(x)|$.
We may assume that $\partial_h^1 v_h(x)$ and $\partial_h^1 w_h(x)$ are nonempty. Assuming that $\partial_h^1 v_h(x)$
is compact and convex, \eqref{conc} follows from \eqref{Brunn-Minkowski}. 

Using the definition and the canonical basis of $\R^d$ one shows that $\partial_h^1 v_h(x)$ is bounded. Thus $\partial_h^1 v_h(x)$ is compact since it can be shown to be a closed set \cite{awanou2019uweakcvg}. The convexity of $\partial_h^1 v_h(x)$ is a consequence of its definition. This concludes the proof.
\end{proof}

\begin{lemma} \label{cone}
Let $C_y(x) = ||y-x||$ denote the cone with vertex $y \in \Omega_h$. Then
$$
M_h[C_y](y) \geq \omega_d >0, 
$$
where $ \omega_d$ is the volume of the closed unit ball.
\end{lemma}

\begin{proof}
We have $C_y(y)=0$ and $p \in \partial_h^1 C_y(y)$ if and only if for all $(e_1,\ldots,e_d) \in V, |p \cdot e_i| \leq ||e_i ||$ for all $i$. Clearly $\partial_h^1 C_y(y)$ contains the closed unit ball with volume $\omega_d$. This concludes the proof.
\end{proof}

\subsubsection{Stability}

By our assumption on $f_h$ we have
\begin{align} \label{st-ass} 
 \sum_{x \in \Omega_h} f_h(x) \leq A,
\end{align}
with $A$ a number independent of $h$. 

\noindent
For $x \in \Omega$ we denote by $d(x,\partial \Omega)$ the distance of $x$ to $\partial \Omega$. For a subset $S$ of $\Omega$, $\diam(S)$ denotes its diameter.

\begin{lemma}  \label{u-stab}
Let $v_h \in \mathcal{C}_h$. Then
$$
\max_{x \in  \tir{\Omega} \cap\mathbb{Z}^d_h} v_h(x) \leq \max_{x \in  \partial \Omega_h}v_h(x).
$$

\end{lemma}

\begin{proof}
Let $x_0 \in \Omega_h$ such that $\max_{x \in  \tir{\Omega} \cap\mathbb{Z}^d_h} v_h(x) = v_h(x_0)$. Assume by contradiction that $v_h(x_0) > \max_{x \in  \partial \Omega_h}v_h(x)$. Let $e \in \mathbb{Z}^d_h$ such that $x_0 \pm e \in \tir{\Omega} \cap\mathbb{Z}^d_h$ with $x_0 + e \in \partial \Omega_h$ or $x_0 - e \in \partial \Omega_h$. We may assume that $x_0 + e \in \partial \Omega_h$. Then by assumption
$
v_h(x_0) > v_h(x_0 + e)
$, and by definition of maximum $v_h(x_0) \geq v_h(x_0 - e)$.
It follows that $\Delta_e v_h(x_0) <0$, contradicting the assumption $v_h \in \mathcal{C}_h$. 
\end{proof}

\noindent
We will need the following two lemmas from \cite{awanou2019uweakcvg}. 
The first is an analogue of \cite[Lemma 1.4.1]{Guti'errez2001} and the second is a discrete version of the Aleksandrov-Bakelman-Pucci's maximum principle \cite[Theorem 8.1]{NguyenThesis}, analogues of which can be found in \cite{Nochetto14} and \cite{KuoTrudinger00}.

\begin{lemma} \label{dmp-01}
Let $v_h, w_h \in \mathcal{U}_h$ such that $v_h \leq w_h$ on $\partial \Omega_h$ and $v_h \geq w_h$ in $\Omega_h$, then
$$
\partial_h v_h (\Omega_h) \subset \partial_h w_h (\Omega_h). 
$$
\end{lemma}

\begin{lemma} \label{d-stab}
Let $u_h \in \mathcal{C}_h$ such that $u_h \geq 0$ on $\partial \Omega_h$. Then for 
$x \in \Omega_h$ 
$$
 u_h(x) \geq - C(d) \bigg[ \diam(\Omega)^{d-1} d(x,\partial \Omega)  M_h[u_h](\Omega_h) \bigg]^{\frac{1}{d}},
$$
for a positive constant $C(d)$ which depends only on $d$. 
\end{lemma}

We also have
\begin{lemma} \label{dmp-011}
Let $v_h, w_h \in \mathcal{U}_h$ such that $v_h \leq w_h$ on $\partial \Omega_h$ and $v_h \geq w_h$ in $\Omega_h$, then
$$
\partial_h^1 v_h (\Omega_h) \subset \partial_h^1 w_h(\Omega_h). 
$$
\end{lemma}
The proof of the above lemma is based on the same principles as the proof of Lemma \ref{dmp-01}.

\begin{thm}
The solution $u_h$ of \eqref{m1h-f} is uniformly bounded.
\end{thm}

\begin{proof}
By Lemma \ref{u-stab}, we have
\begin{equation} \label{stability1}
u_h(x) \leq  \max_{x \in \partial \Omega_h} g(x).
\end{equation}
By Lemma \ref{d-stab} 
$$
 u_h(x) -  \min_{x \in \partial \Omega_h} g(x) \geq - C(d) \bigg[ \diam(\Omega)^{d-1} d(x,\partial \Omega)  M_h[u_h](\Omega_h) \bigg]^{\frac{1}{d}}.
$$
Since $u_h$ solves \eqref{m1h-f}, by \eqref{st-ass} and \eqref{comparison3}
\begin{align*}
A \geq  \sum_{x \in \Omega_h} f_h(x) = \sum_{x \in \Omega_h} M_h^1 [u_h](x) \geq \sum_{x \in \Omega_h} M_h [u_h](x)
\geq M_h[u_h](\Omega_h).
\end{align*}
In addition $d(x,\partial \Omega)  \leq \diam(\Omega)$. We conclude that $u_h(x) \geq  \min_{x \in \partial \Omega_h} g(x) - C$, for a constant $C$. Combined with \eqref{stability1}, we have shown that the solution $u_h$ of \eqref{m1h-f} is uniformly bounded.
\end{proof}

\subsubsection{Unicity}

The next lemma is an analogue of \cite[Theorem 1.4.6]{Guti'errez2001}.

\begin{lemma} \label{dmp-02}
Let $v_h, w_h \in \mathcal{U}_h$ such that
$$
M_h^1[v_h](x) \leq M_h^1[w_h](x), \forall x \in \Omega_h.
$$
Then
$$
\min_{x \in  \tir{\Omega} \cap\mathbb{Z}^d_h}(v_h(x) - w_h(x)) = 
\min_{x \in  \partial \Omega_h}(v_h(x) - w_h(x)). 
$$
\end{lemma}

\begin{proof}
Let $a=\min_{x \in \Omega_h \cup \partial \Omega_h}(v_h(x) - w_h(x))$ and $b=\min_{x \in  \partial \Omega_h}(v_h(x) - w_h(x))$.

Assume that $a<b$ and let $x_0 \in  \Omega_h$ such that $a=v_h(x_0) - w_h(x_0)$. Choose $\delta >0$ such that 
$\delta (\diam \Omega) < (b-a)/2$ and define
$$
z(x) = w_h(x) + \delta ||x-x_0|| + \frac{b+a}{2}.
$$
Let $G= \{\, x \in  \tir{\Omega} \cap\mathbb{Z}^d_h \text{ such that } v_h(x) < z(x)\, \}$. It is easy to verify that $x_0 \in G$. We claim that $G \cap \partial \Omega_h = \emptyset$.  

Let $x \in G \cap \partial \Omega_h$. We have $v_h(x) - w_h(x) \geq b$ and so
\begin{align*}
z(x) & \leq v_h(x) - b+ \delta ||x-x_0|| + \frac{b+a}{2} = v_h(x) + \delta ||x-x_0|| - \frac{b-a}{2} < v_h(x),
\end{align*}
by the assumption on $\delta$. We define
$$
\partial G = \{ \, x \in \tir{\Omega} \cap\mathbb{Z}^d_h \text{ such that } x \notin G \, \}.
$$
We have $z \leq v_h$ on $\partial G$ and $z>v_h$ in $G$. By Lemma  \ref{dmp-011} we obtain 
$\partial_h^1 z (G) \subset \partial_h^1 v_h (G)$. And thus by Lemmas \ref{concave} and \ref{cone}
\begin{align*}
M_h^1[v_h] (G) & \geq M_h^1[z] (G) \geq M_h^1[w_h] (G)  + M_h^1[\delta ||x-x_0||] (G) \\
& \geq M_h^1[w_h] (G)  + M_h^1[\delta ||x-x_0||] (x_0)= M_h^1[w_h] (G)  + \delta^d\omega_d.
\end{align*}
This gives a contradiction. 
\end{proof}

We obtain the following easy consequence of Lemma \ref{dmp-02}.
\begin{thm}
The solution $u_h$ of \eqref{m1h-f} is unique.
\end{thm}

\begin{proof} Let $v_h$ and $w_h$ be two solutions of \eqref{m1h-f}. We have $M_h^1 [v_h](d_l) = M_h^1 [w_h](d_l), l=1,\ldots,N$. For $x \in \Omega_h,  x \notin \{ \, d_1,\ldots,d_N \, \}$, we have $\mathcal{M}_h [v_h](x) = \mathcal{M}_h [w_h](x)=0$. By \eqref{compare-all}, we get $M_h^1 [v_h](x) = M_h^1 [w_h](x)$ for all $x \in \Omega_h$. Unicity then follows from  Lemma \ref{dmp-02}.
\end{proof}

\subsubsection{Existence}

We show that minimizers of a convex functional over a convex set solve \eqref{m1h-f}. For $v_h \in \mathcal{U}_h$ and $i=1,\ldots,d$ we consider the first order difference operator defined by 
\begin{align*}
\partial^i_{-} v_h(x) & \coloneqq \frac{ v_h(x)-v_h(x-he_i) } {h}, x \in \Omega_h, 
\end{align*}
and the convex functional
$$
J_h(v_h) =  \sum_{x \in \Omega_h^{}} h^d || D_h v_{h}(x)||^2, 
$$
where $D_h v_{h}$ is the vector of backward finite differences of the mesh  
function $v_h$, i.e.
$$
D_h v_{h}(x) = (\partial^i_- v_h(x))_{i=1,\ldots,d}.
$$ 
Let
\begin{align} \label{s-h-20}
\begin{split}
S_h & = \{\, v_h \in \mathcal{C}_h, v_h= g_h \, \text{on} \,  \partial \Omega_h, 
\, \text{and} \,
(M_h^1[v_h](x) )^{\frac{1}{d}} \geq  f_h(x)^{\frac{1}{d}}, x \in \Omega_h \, \}.
\end{split}
\end{align}
We seek a minimizer of $J_h$ over $S_h$. Note that for $x \in \Omega_h \setminus \cup_{l=1}^N \{ \, d_l \, \}$, we have 
$M_h^1[v_h](x)=0$ when $v_h$ solves \eqref{m1h-f}. Therefore it is enough to consider the operator $M_h^1[v_h]$.

\begin{lemma} \label{g-case}
The set $S_h$ is  convex  and nonempty.
\end{lemma}
\begin{proof} The convexity of $S_h$ follows from Lemma \ref{concave}. 

For each $y$ in $\Omega_h$, let $q_y$ be a cone such that 
$M_h^1[q_y](y) \geq  f_h(y)$. For example, we may define 
$q_y$ by
$$
q_y(x) =  \bigg( \frac{f_h(y)}{\omega_d} \bigg)^{\frac{1}{d}} C_y(x).
$$
Put $\hat{q} = \sum_{y \in \Omega_h} q_y$. 
Since $g$ is bounded on $\partial \Omega$, we can find a number $\kappa$ such that $\hat{q} - \kappa \leq g$ on $\partial \Omega$. We define $w_h \in \mathcal{U}_h$ by 
\begin{align*}
w_h(x) & = \hat{q}(x)-\kappa, x \in \Omega_h \\
w_h & = g \text{ on } \partial \Omega_h.
\end{align*}
We claim that $w_h \in S_h$. 

For $x \in \Omega_h$ and $e, p$ such that $x + e, x-p \in \tir{\Omega}_h$, either $w_h(x+e)=\hat{q}(x+e)-\kappa$ or $w_h(x+e)=g(x+e) \geq \hat{q}(x+e)-\kappa$. Similarly $w_h(x-p) \geq \hat{q}(x-p)-\kappa$. We conclude using the convexity of $\hat{q}$ that
$$
\Delta_e w_h(x) \geq \hat{q}(x+e)-\kappa -2 w_h(x) + \hat{q}(x-e)-\kappa
= \Delta_e  \hat{q}(x) \geq 0.
$$
Thus $w_h \in \mathcal{C}_h$. Next, we prove that $\partial_h^1 \hat{q}(x) \subset \partial_h^1 w_h(x)$ for $x \in \Omega_h$. 

Since $w_h = \hat{q}$ up to a constant on $\Omega_h$, we only need to check that for $p \in \partial_h^1 \hat{q}(x)$ we have
$w_h(x) - w_h(x-e) \leq p \cdot e \leq w_h(x+e) -w_h(x)$ when $x \pm e \in \partial \Omega_h$.

\noindent
Let $p \in \partial_h^1 \hat{q}(x)$ such that $x+e \in \partial \Omega_h$. We have
\begin{align*}
p \cdot e & \leq  \hat{q}(x+e) -  \hat{q}(x) = \hat{q}(x+e) -\kappa -w_h(x) \leq g(x+e) -w_h(x)\\
& = w_h(x+e) -w_h(x).
\end{align*}
Similarly, if  $x-e \in \partial \Omega_h$ and $p \cdot e \geq \hat{q}(x) - \hat{q}(x-e)$ we obtain $p \cdot e \geq w_h(x) - w_h(x-e)$. We conclude that $M_h^1[w_h](x) \geq M_h^1[\hat{q}](x)$. Therefore by the Brunn-Minkowski inequality \eqref{Brunn-Minkowski}, 
$M_h^1[w_h](x)^{1/d} \geq \sum_{y \in \Omega_h} M_h^1[q_y](x)^{1/d} \geq  f_h(x)^{1/d}$. 
This concludes the proof. 
\end{proof}

\begin{thm}
The functional $J_h$ has a  minimizer $u_h$ in $S_h$  
and $u_h$ solves the finite difference equation \eqref{m1h-f}.
\end{thm}

\begin{proof} Since $J_h$ is  
convex and $S_h$ is nonempty, it follows that the functional $J_h$  
has a minimizer $u_h$ on the convex set $S_h$.

We now show that $u_h$ solves the finite difference system \eqref{m1h-f}. To this end, 
it suffices to show that 
$$M_h^1[u_h]=  f_h \text{ on } \Omega_h.$$
Let us assume to the contrary that there exists $x_0 \in 
\Omega_h$ such that
\begin{equation} \label{strict-x0}
M_h^1[u_h](x_0) >  f_h(x_0) \geq 0.
\end{equation}
If there were a direction $e$ such that $\Delta_e u_h(x_0)=0$, we would have $\partial_h^1u_h(x_0)$ contained in the hyperplane $p \cdot e = u_h(x_0+e)-u_h(x_0)=u_h(x_0)-u_h(x_0-e)$, and hence $M_h^1[u_h](x_0)=0$ contradicting 
\eqref{strict-x0}. We conclude that for all $e \in \mathbb{Z}^d$ such that $x_0\pm h e \in \tir{\Omega}$, $ \Delta_e u_h(x_0) >0$. Let $$\epsilon_0 = \inf \{ \,   \Delta_e u_h(x_0), e \in \mathbb{Z}^d, x_0 \pm h e \in \tir{\Omega} \}.$$
We recall that $M_h^1[u_h](x)$ is the volume of a polygon since it is the volume of a domain obtained as an intersection of half-spaces, e.g. $ p \cdot e \leq u_h(x+e) - u_h(x)$  and $p \cdot e \geq u_h(x) - u_h(x-e)$. 
Moreover $\partial_h^1 u_h(x)$ is bounded as its volume is bounded by $M_h^0[u_h](x)$. 
The vertices of the polygon have coordinates linear combinations of the values $u_h(y), y \in \Omega_h$. It is known that  the volume of a polygon is a polynomial function, hence a continuous function, of the coordinates of its vertices \cite{Allgower86}. Thus the mapping $E: u_h(x_0) \to M_h^1[u_h](x_0)$ is continuous and by
\eqref{strict-x0}, with $r_0=u_h(x_0)$, $E(r_0) >  f_h(x_0)$. Therefore there exists $\epsilon_1 >0$ such that for
$|r-r_0|< \epsilon_1$, we have $E(r) >  f_h(x_0)$.  

Finally, put $\epsilon=\min( \epsilon_0, \epsilon_1)$. We define $w_h$ by
$$
w_h(x) = u_h(x), x \neq x_0, w_h(x_0) = u_h(x_0) + \frac{\epsilon}{4}.
$$
By construction $w_h= g_h \, \text{on} \,  \partial \Omega_h$. For 
$x \neq x_0$ either $\Delta_e w_h(x)= \Delta_e u_h(x)$ or $\Delta_e w_h(x)= \Delta_e 
u_h(x) + \epsilon/4$. Moreover 
$\Delta_e w_h(x_0)= \Delta_e u_h(x_0)- \epsilon/2 \geq \epsilon_0- \epsilon/2 
\geq \epsilon/2 >0$ by the definition of $\epsilon$. We conclude that $w_h \in 
\mathcal{C}_h$.

Also by construction, $M_h^1[w_h](x_0)=E(r_0+\epsilon/4)>f_h(x_0)$. We claim that for $x \neq x_0$ $M_h^1[w_h](x) \geq M_h^1[u_h](x)$. Let $p \in \R^d$ such that
$u_h(x) - u_h(x-e) \leq p \cdot e \leq u_h(x+e) - u_h(x)$. Either $u_h(x+e) = w_h(x+e)$ or $u_h(x+e) = w_h(x+e) - \epsilon/4$. This gives $p \cdot e \leq w_h(x+e) - w_h(x)$. Similarly $w_h(x) - w_h(x-e) \leq p \cdot e$. This proves the claim.
We conclude that $M_h^1[w_h](x) \geq f_h(x)$ for all $x \in  \Omega_h$.

It remains to show that $J_h(w_h) < J_h(u_h)$. Let $\Omega_{x_0}$ denote 
the subset of  $\Omega_h$ consisting in $x_0$ and the points $x_0 + h r_j, 
j=1,\ldots,d$ at which $D_h u_h$ is defined.
We have
\begin{align*}
J_h(w_h) & = h^d \sum_{x \notin \Omega_{x_0}} ||D_h u_h(x)||^2 + h^d  
||D_h w_h(x_0)||^2 +  \sum_{j=1}^d  ||D_h w_h(x_0+ h r_j)||^2 \\
J_h(w_h) & = h^d \sum_{x \notin \Omega_{x_0}} ||D_h u_h(x)||^2 + h^{d-2}
\sum_{i=1}^d (w_h(x_0) - w_h(x_0-h r_i))^2 \\
& \quad + h^{d-2} \sum_{j=1}^d \sum_{i=1}^d (w_h(x_0+h r_j) - 
w_h(x_0 +h r_j - h r_i))^2 \\
& = h^d \sum_{x \notin \Omega_{x_0}} ||D_h u_h(x)||^2  + h^{d-2} 
\sum_{j=1}^d \sum_{i=1 \atop i \neq j}^d (w_h(x_0+h r_j) - w_h(x_0+h r_j - h r_i 
))^2 \\
& \quad +  h^{d-2} \sum_{i=1}^d (w_h(x_0) - w_h(x_0-h r_i))^2 + 
(w_h(x_0+h r_i) - w_h(x_0 ))^2.
\end{align*}
However
\begin{multline*}
\sum_{i=1}^d (w_h(x_0) - w_h(x_0-h r_i))^2 + (w_h(x_0+h r_i) - w_h(x_0))^2 = \\
\sum_{i=1}^d \big(u_h(x_0) - u_h(x_0-h r_i) + \frac{\epsilon}{4}\big)^2 + \big(u_h(x_0+h r_i) - u_h(x_0)-\frac{\epsilon}{4}\big)^2 = \\
\sum_{i=1}^d (u_h(x_0) - u_h(x_0-h r_i))^2 + (u_h(x_0+h r_i) - u_h(x_0))^2 + 
\frac{d \epsilon^2}{8} - \frac{h^2 \epsilon}{2} \Delta_h u_h(x_0),
\end{multline*}
where
$$
\Delta_h v_h(x) = \sum_{i=1}^d \frac{v_h(x+hr_i) -2 v_h(x) + v_h(x-h r_i)}{h^2}.
$$
Thus, since for $i\neq j, w_h(x_0+h r_j) - w_h(x_0+h r_j - h r_i) = u_h(x_0+h r_j) - u_h(x_0+h r_j - h r_i)  $, and by our choice of $\epsilon$,
\begin{align*}
J_h(w_h) & = J_h(u_h) + \frac{d h^{d-2} \epsilon^2}{8} -   \frac{ h^d \epsilon}{2} 
\Delta_h u_h(x_0) = J_h(u_h) + \frac{h^{d-2} \epsilon}{2}(  \frac{d \epsilon}{4} - 
h^2 \Delta_h u_h(x_0)) \\
& <  J_h(u_h),
\end{align*}
since $ \Delta_e u_h(x_0) \geq 
\epsilon_0$ and thus $h^2 \Delta_h u_h(x_0)) \geq d \epsilon_0 \geq d \epsilon > d 
\epsilon /4$. This contradicts the assumption that $u_h$ is 
a minimizer and concludes the proof.
\end{proof}

\section{Convergence analysis of the scheme obtained through a partial discrete normal mapping} \label{cvg-analysis}

This section is devoted to the convergence of the solution $u_h$ of \eqref{m1h-f} to the Aleksandrov solution $u$ of \eqref{m1}.

Let $\mathcal{T}_h$ denote a triangulation of the convex hull of $\tir{\Omega}_h$ with vertices in $\tir{\Omega}_h$. 
We denote by $I (u_h)$ the piecewise linear continuous function which is equal to $u_h$ on the set of vertices $\tir{\Omega}_h$. We make the assumption that the triangulation $\mathcal{T}_h$ is chosen such that the interpolant $I(u_h)$ is piecewise linear along the coordinate axes, i.e. the line segments in $\tir{\Omega}$ through $x \in \Omega_h$ and directions $e \in \{ \, r_1,\ldots,r_d\, \}$ the canonical basis of $\R^d$.

\begin{defn}
We say that $u_h$ converges to a function $u$ on $\Omega$ uniformly on compact subsets of $\Omega$ if and only if $I(u_h)$  converges uniformly on compact subsets of $\Omega$ to $u$.

\end{defn}

\begin{thm} \label{main1}
Let $u_h$ solve \eqref{m1h-f}. There is a subsequence $u_{h_k}$ which converges uniformly on compact subsets to a convex function $v \in C(\tir{\Omega})$ such that
\begin{align} 
\begin{split}
\det D^2 v & \leq \sum_{l=1}^N c_l \delta_{d_l} \, \text{in} \, \Omega\\
 v &=g \, \text{on} \, \partial \Omega,
 \end{split}
\end{align}

\end{thm}

\begin{proof}The family $u_h$ is a uniformly bounded sequence of discrete convex functions. Moreover $u_h=g$ on $\partial \Omega$ and $g \in C(\partial \Omega)$ can be extended to a  convex function $\tilde{g} \in C(\tir{\Omega})$. In addition, by 
\eqref{comparison3} and \eqref{st-ass}
$$
M_h(\Omega_h) \leq \sum_{x \in \Omega_h} M_h [u_h](x) \leq \sum_{x \in \Omega_h} M_h^1 [u_h](x) =  \sum_{x \in \Omega_h} f_h(x) \leq A.
$$
It is proven in \cite{awanou2019uniform} that there is a subsequence $u_{h_k}$ which converges uniformly on compact subsets to a convex function $v \in C(\tir{\Omega})$ such that $v=g$ on $\partial \Omega$. 
It is also proven in \cite{awanou2019uweakcvg} that $M_h[u_h]$ defines a Borel measure which converge weakly to $M[v]$. Since by \eqref{compare-all}, we have
$$
M_h[u_h] (x) \leq  f_h(x), x \in \Omega_h, 
$$
as an equation in measures, we obtain $\det D^2 v  \leq \sum_{l=1}^N c_l \delta_{d_l}$.
\end{proof}

To complete the proof we need additional notions. Given $u: \Omega \to \R$, the local subdifferential of $u$ is given by
\begin{align*} 
\partial_l u (x_0) & = \{ \, p \in \R^d: \exists \, \text{a neighborhood} \, U_{x_0} \, \text{of} \, x_0 \, \text{such that} \\
& \qquad   u(x) \geq u(x_0) + p \cdot (x-x_0), \, \text{for all} \, x \in U_{x_0}\,\}.
\end{align*}
Clearly for all $x_0 \in \Omega$ we have $\partial u (x_0) \subset \partial_l u (x_0)$. Moreover
\begin{lemma}[\cite{Guti'errezExo} Exercise 1] \label{lem-exo}
If $\Omega$ is convex and $u$ is convex on $\Omega$, then  $\partial u (x) = \partial_l u (x)$ for all $x \in \Omega$.
\end{lemma}

We recall that for a family of sets $A_k$
$$
\limsup_k A_k  = \cap_{n} \cup_{k \geq n} A_k.
$$

\begin{lemma} \label{lem-weak01}
Assume that $u_h \to v$ uniformly on compact subsets of $\Omega$, with $v$ convex and continuous. Then for $K \subset \Omega$ compact and any sequence $h_k \to 0$
$$
\limsup_{h_k \to 0} \partial_{h_k}^1 u_{h_k}(K) \subset \partial v(K).
$$
\end{lemma}

\begin{proof}
Let 
$$
p \in \limsup_{h_k \to 0} \partial_{h_k}^1 u_{h_k}(K)= \cap_{n} \cup_{k \geq n} \partial_{h_k} u_{h_k}(K).
$$
Thus for each $n$, there exists $k_n$ and $x_{k_n} \in K \cap \mathbb{Z}_h^d$ such that $p \in  \partial_{h_{k_n}}^1 u_{h_{k_n}}(x_{k_n})$. Let $x_j$
denote a subsequence of $x_{k_n}$ converging to $x_0 \in K$. 

Let $B_{\epsilon}(x_0)$ denote the ball of center $x_0$ and radius $\epsilon$ in the maximum norm.
We choose $\epsilon >0$ such that $B_{\epsilon}(x_0) \subset \Omega$. Let $z \in B_{\epsilon/4}(x_0)$ and $z_h \in B_{\epsilon/4}(x_0) \cap  \mathbb{Z}_h^d$ such that $z_h \to z$.

We have for $j$ sufficiently large $||x_j-x_0|| \leq \epsilon/8$. With $e=z_h-x_j, x_j+e=z_h$ while $x_j-e=2 x_j-z_h \in B_{\epsilon/4}(x_0)$ as $||2 x_j-z_h-x_0|| = || 2(x_j-x_0) + (x_0-z_h)|| \leq \epsilon/4$. That is $x_j \pm e \in \Omega \cap \mathbb{Z}_h^d$. 

Since $p \in \partial_{h_j}^1 u_{h_j}(x_j)$ for all $j$,  
\begin{equation} \label{p-in-}
u_{h_j}(z) \geq u_{h_j}(x_j) +p \cdot (z_{h_j}-x_j). 
\end{equation}
Next, note that
\begin{align*}
|u_{h_j}(x_j) - v(x_0)| \leq |u_{h_j}(x_j) - v(x_j)| + |v(x_j) - v(x_0)|.
\end{align*}
By the convergence of $x_j$ to $x_0$ 
and the uniform convergence of $u_h$ to $v$, we obtain $u_{h_j}(x_j) \to v(x_0)$ as $h_j \to 0$. Similarly $u_{h_j}(z) \to v(z)$ as $h_j \to 0$.

Taking pointwise limits in \eqref{p-in-}, we obtain  
$$
v(z) \geq v(x_0) + p \cdot (z-x_0)  \ \forall z \in B_{\frac{\epsilon}{4}}(x_0).
$$
We conclude that $p \in \partial_l v(K)$, the image of $K$ by the local subdifferential of $v$, and thus $p \in \partial v(K)$ by Lemma \ref{lem-exo}, since $v$ is convex and $\Omega$ convex.
\end{proof}

\begin{thm}
The limit convex function $v$ given by Theorem \ref{main1} satisfies
$$
M[v](d_l) =  c_l, l=1,\ldots,N.
$$
\end{thm}

\begin{proof}
We recall that by assumption, the points $d_l$ are assumed to be mesh points for all $l$. 
The set $\partial_h^1 u_h(\{\, d_l \, \})$ is a closed polygon and hence Lebesgue measurable. 
It follows from Lemma \ref{lem-weak01} with $K=\{\, d_l \, \}$, see also \cite{awanou2019uweakcvg} for a detailed argument which uses only set properties and properties of the integral, that 
$$
\limsup_{h_k \to 0} M_{h_k}^1[u_{h_k}](d_l) \leq M[v](d_l).
$$ 
By Theorem \ref{main1}  $M[v](d_l) \leq c_l$. But $ M_h^1[u_h](d_l) =  f_h(d_l)$ and by assumption on $f_h$, we have $ f_h(d_l)$ 
converges to $c_l$. We conclude that 
$$
\limsup_{h_k \to 0} M_{h_k}^1[u_{h_k}](d_l) = \limsup_{h_k \to 0} f_{h_k}(d_l) = c_l \leq M[v](d_l) \leq c_l. 
$$\end{proof}

\begin{thm} \label{main2}
The solution $u_h$ of \eqref{m1h-f} converges uniformly on compact subsets to the solution $u$ of \eqref{m1}.
\end{thm}

\begin{proof}
It follows from Theorems \ref{main1} and \ref{main2}, that there is a subsequence which converges uniformly on compact subsets to a convex function $v \in C(\tir{\Omega})$ which solves \eqref{m1}. By unicity of the solution of the latter, the whole family must converge to $u$.
\end{proof}

\section{Concluding remarks}

Wide stencils are implemented using only a small number of directions $e \in \mathbb{Z}_h^d$. For increased accuracy, it is sometimes necessary to use a significant number of mesh points on  $\partial \Omega$ as follows.

For $x \in \Omega_h$ and $ e \in \mathbb{Z}^d$ let
$$
h^e_x = \sup \{ \, r h, r \in [0,1] \ \text{ and } \ x+rh e \in \tir{\Omega} \, \}.
$$
Now, let
\begin{equation} \label{boundary-domain}
\partial \Omega_h = \{ \, x \in \partial \Omega,  x= y + h^e_y e \ \text{for} \ y \in \Omega_h  \ \text{and} \ e \in \mathbb{Z}^d   \,  \}.
\end{equation}
As in \cite{Mirebeau15}, for $e \in \mathbb{Z}^d, v_h \in \mathcal{U}_h$ and $x \in \Omega_h$ we now define
$$
\Delta_e v_h (x) = \frac{2}{h^e_x + h^{-e}_x} \bigg( \frac{v_h(x+ h^e_x e ) - v_h(x)}{h^e_x}  + \frac{v_h(x- h^{-e}_x e ) - v_h(x)}{h^{-e}_x} \bigg),
$$
and consider the following versions of the standard wide stencil discrete operator
$$
\mathcal{M}_h[v_h](x) = \inf_{ (e_1,\ldots,e_d) \in V } \prod_{i=1}^d \frac{\Delta_{e_i} v_h (x)}{|| e_i ||^2}, x \in \Omega_h,
$$
and the discretization of the subdifferential
\begin{align*}
\partial_h u_h(x) = \{ \, p \in \R^d, h^{-e}_x p \cdot e \geq u_h(x) - u_h(x-h^{-e}_xe) \, \forall  e \in \mathbb{Z}^d  \, \}.
\end{align*}
We may assume that the Dirac masses are far from $\partial \Omega_h$, so no modification of $\partial_h^1 u_h (x )$ is necessary. What is essential is that the inequality $M_h[v_h](x) \leq C \mathcal{M}_h[v_h](x)$ still holds.
We leave the details to the interested reader. 

With the results of this paper, it should be possible through a perturbation analysis to analyze for the Dirichlet problem the discretization used in \cite{BenamouFroese2017}. This also applies to \eqref{sub-1h-sided} when only a limited set of directions are used, as we noted in section \ref{quad}. 

\section*{Acknowledgments}

The author was partially supported by NSF grants DMS-1319640 and  DMS-1720276. The author would like to thank the Isaac Newton Institute for Mathematical Sciences, Cambridge, for support and hospitality during the programme ''Geometry, compatibility and structure preservation in computational differential equations'' where part of this work was undertaken. Part of this work was supported by EPSRC grant no EP/K032208/1.






\end{document}